\documentclass{amsart}

\usepackage{fullpage}
\usepackage{amssymb,amsmath,amsfonts, amscd}
\usepackage{array}
\usepackage{amsmath}
\usepackage{amssymb}

\newtheorem{thm}{Theorem}[section]

\newtheorem{cor}[thm]{Corollary}
\newtheorem{lem}[thm]{Lemma}
\newtheorem{defn}[thm]{Definition}
\newtheorem{prop}[thm]{Proposition}

\newenvironment{proofcite}[1]{\noindent{\bf Proof of #1.\,}}{\hfill$\Box$ \\}

\newcommand{\E}{{\mathcal{E}}}
\newcommand{\A}{{\mathcal{A}}}
\newcommand{\B}{{\mathcal{B}}}
\newcommand{\F}{{\mathcal{F}}}
\newcommand{\K}{{\mathcal{K}}}
\newcommand{\hh}{{\mathcal{H}}}
\newcommand{\N}{{\mathbb{N}}}
\newcommand{\p}{{\bf p}}
\newcommand{\bfd}{{\bf d}}
\newcommand{\textdef}{\textbf}
\newcommand{\dist}{{\rm dist}}
\newcommand{\forb}{{\rm Forb}}

\newcommand{\Aset}{\overleftrightarrow{\mathcal{A}}}
\newcommand{\Pset}{\mathcal{P}}
\newcommand{\nonarrow}{\bigcirc}
\newcommand{\unarrow}{-}

\begin{document}
\title{A version of Szemer\'edi's regularity lemma for multicolored graphs and directed graphs that is suitable for induced graphs}
\author{Maria Axenovich}
\address{Department of Mathematics, Iowa State University, Ames, Iowa 50011}
\email{axenovic@iastate.edu}
\thanks{This author's research partially supported by NSF grant DMS-0901008 and NSA grant H-98230-09-1-0063.}
\author{Ryan Martin}
\address{Department of Mathematics, Iowa State University, Ames, Iowa 50011}
\email{rymartin@iastate.edu}
\thanks{This author's research partially supported by NSF grant DMS-0901008 and by an Iowa State University Faculty Professional Development grant.}

\maketitle

\begin{abstract}
In this manuscript we develop a version of Szemer\'edi's regularity lemma that is suitable for analyzing multicolorings of complete graphs and directed graphs.  In this, we follow the proof of Alon, Fischer, Krivelevich and M. Szegedy [\textit{Combinatorica} \textbf{20}(4) (2000) 451--476] who prove a similar result for graphs.

The purpose is to extend classical results on dense hereditary properties, such as the speed of the property or edit distance, to the above-mentioned combinatorial objects.
\end{abstract}

\section{Introduction}
We develop a version of Szemer\'edi's regularity lemma that is suitable for analyzing multicolorings of complete graphs and directed graphs. In proving our theorems we use as our guide the proof given by Alon, Fischer, Krivelevich and M. Szegedy~\cite{AFKSz} which proves a similar theorem in the case of graphs.  Their idea is, when given a graph, $G$, they find an induced subgraph $G'$ and two equipartitions, $\A$ of $V(G)$ and $\A'$ of $V(G')$.  The partitions $\A$ and $\A'$ have the same number of parts. Each part of $\A'$ is large and contained in some part of $\A$, each pairwise density of the parts in $\A'$ is close to the density of the corresponding pair in $\A$, but \textbf{all} pairs in $\A'$ are regular. Our goal is to find an induced copy of $H$ in $G$.  If enough of the pairs of parts in $\A$ have a sufficiently large density, we can apply the regularity lemma and Ramsey's theorem inside each of the parts of $\A'$.  A slicing lemma ensures that the resulting subclusters (we call them miniclusters) are ready to witness the embedding of a graph $H$.

In fact, this approach works for any combinatorial object that has a sufficiently similar type of regularity lemma.~\\

\noindent\textbf{Outline of the paper:} In the following subsections, we give basic definitions and the results we need for the graph version (Section~\ref{sec:graphs}), the multicolor version (Section~\ref{sec:multicol}) and the digraph version (Section~\ref{sec:digraph}).  Section~\ref{sec:main} gives the main result.  In Section~\ref{sec:proofs}, we prove our main results for multicolored graphs and for directed graphs simultaneously -- the main machinery depends very little on the combinatorial object to be studied.  In Section~\ref{sec:apps}, we apply our result to a specific problem related to edit distance.

\begin{defn}
A partition $\A=\{V_i : 1\leq i\leq k\}$ is an \textdef{equipartition} of a finite set if $|V_i|$ and $|V_{i'}|$ differ by at most 1 for all $1\leq i<i'\leq k$.  A \textdef{refinement} of $\A$ is a partition $\B=\{V_{i,j_i} : 1\leq i\leq k, 1\leq j_i\leq\ell_i\}$ such that $V_i=\bigcup_{j=1}^{\ell_i}A_{i,j_i}$ for $i=1,\ldots,k$.  The number of parts of a partition is its \textdef{order}.
\end{defn}

Just to ensure some technicalities, we prove that every equipartition can be refined into an equipartition.
\begin{prop}
Let $\A=\{V_i : 1\leq i\leq k\}$ be an equipartition of a finite set and let $\ell$ be a positive integer, $\ell\leq |V_i|$, $i=1,\ldots,k$. There exists a refinement of $\A$ into $k\ell$ parts that is an equipartition.
\end{prop}

\begin{proof}
If all the $V_i$ are the same size, it is clear that equipartitioning each will result in the equipartition we seek. Suppose the sizes of each $V_i$ are $s$ and $s-1$ such that $s=q\ell+r$ for $r\in\{0,\ldots,\ell-1\}$.  It suffices to show $\lceil s/\ell\rceil$ and $\lfloor (s-1)/\ell\rfloor$ differ by at most one.

If $r\neq 0$, then $\lceil s/\ell\rceil=q+1$ and $\lfloor (s-1)/\ell\rfloor=q$.  If $r=0$, then $\lceil s/\ell\rceil=q$ and $\lfloor (s-1)/\ell\rfloor=q-1$.
\end{proof}

\subsection{Graph version}
\label{sec:graphs}
A graph $G$ is a pair $(V,E)$ where $V$ is a finite vertex set and $E\subseteq\binom{V}{2}$.

For disjoint vertex sets $V_i$, $V_j$, we denote $e(V_i,V_j)$ to be number of edges with one endpoint in $V_i$ and the other in $V_j$.  The \textdef{density} of $(V_i,V_j)$ is
$$ d(V_i,V_j):=\frac{e(V_i,V_j)}{|V_i||V_j|} . $$
The \textdef{density vector} of the pair $(V_i,V_j)$ is simply
$$ \bfd(V_i,V_j):=\left(d(V_i,V_j),1-d(V_i,V_j)\right) . $$

We say the pair $(V_i,V_j)$ is a \textdef{$\gamma$-regular pair} if  $V_i'\subset V_i$ and $V_j'\subset V_j$ such that both $|V_i'|\geq \gamma |V_i|$ and $|V_j'|\geq \gamma |V_j|$, then $|d(V_i',V_j')-d(V_i,V_j)|\leq\gamma$.

A partition $(V_1,\ldots,V_k)$ of the vertex set of $G$, a graph on $n$ vertices, is said to be a \textdef{$\gamma$-regular partition} if each of the following holds:
\begin{itemize}
   \item $\left||V_i|-|V_j|\right|\leq 1$ for all $i,j\in\{1,\ldots,k\}$.
   \item All but at most $\gamma k^2$ of the pairs $(V_i,V_j)$, $1\leq i<j\leq k$ are $\gamma$-regular.
\end{itemize}

A version of Szemer\'edi's lemma says the following:
\begin{thm}[Szemer\'edi~\cite{Sz}]\label{thm:reglem:graph}
   For every $m$ and $\epsilon>0$, there exists an integer $M=M(m,\epsilon)$ with the following property.

   If $G$ is a graph with $n\geq M$ vertices, and $\A$ is an equipartition of the vertex set of $G$ of order at most $m$, then there exists a refinement $\B$ of $\A$ of order $k$, where $m\leq k\leq M$, which is $\epsilon$-regular.
\end{thm}

There are two important lemmas cited by Alon, et al.~\cite{AFKSz} which permit discussion of graph embedding.  They have been presented and reproven many times, we give the statements here.  The titles ``Slicing lemma'' and ``Embedding lemma'' can be found in the literature.
\begin{lem}[Slicing lemma]\label{lem:slicing:graph}
   If $(A,B)$ is a $\gamma$-regular pair with density $\delta$ and $A'\subset A$ and $B'\subset B$ satisfy $|A'|\geq\epsilon |A|$ and $|B'|\geq\epsilon |B|$ for some $\epsilon\geq\gamma$, then $(A',B')$ is a $(\max\{2,\epsilon^{-1}\}\gamma)$-regular pair with density at least $\delta-\gamma$ and at most $\delta+\gamma$.
\end{lem}

\begin{lem}[Embedding lemma]\label{lem:embed:graph}
\newcommand{\gamembedgraph}{\gamma_{\ref{lem:embed:graph}}}
\newcommand{\delembedgraph}{\delta_{\ref{lem:embed:graph}}}
   For every $0<\eta<1$ and positive integer $k$ there exist $\gamma=\gamembedgraph(\eta,k)$ and $\delta=\delembedgraph(\eta,k)$ with the following property. \\
   \indent Suppose that $H$ is a graph with vertices $v_1,\ldots,v_k$, and that $V_1,\ldots,V_k$ is a $k$-tuple of disjoint vertex sets such that, for every $1\leq i<i'\leq k$, the pair $(V_i,V_{i'})$ is $\gamma$-regular, with density at least $\eta$ if $v_iv_{i'}$ is an edge of $H$ and with density at most $1-\eta$ if $v_iv_{i'}$ is not an edge of $H$.  Then, at least $\delta\prod_{i=1}^k|V_i|$ of the $k$-tuples $w_1\in V_1,\ldots,w_k\in V_k$ span (induced) copies of $H$ where each $w_i$ plays the role of $v_i$.
\end{lem}

\subsection{Multicolor graph version}
\label{sec:multicol}
We call an \textdef{$r$-graph} on $n$ vertices a pair $(V,c)$ where $V$ is a set of size $n$ and $c:\binom{V}{2}\rightarrow\{1,\ldots,r\}$ is a function known as the \textdef{coloring} of the edge set.

For disjoint vertex sets $V_i$, $V_j$ and a color $\rho\in\{1,\ldots,r\}$, we denote $e_{\rho}(V_i,V_j)$ to be number of edges with one endpoint in $V_i$ and the other in $V_j$ and with color $\rho$.  The \textdef{$\rho$-density} of $(V_i,V_j)$ is
$$ d_{\rho}(V_i,V_j):=\frac{e_{\rho}(V_i,V_j)}{|V_i||V_j|} . $$
The \textdef{density vector} of the pair $(V_i,V_j)$ is simply
$$ \bfd(V_i,V_j):=\left(d_1(V_i,V_j),\ldots,d_r(V_i,V_j)\right) . $$

We say the pair $(V_i,V_j)$ is a \textdef{$\gamma$-regular pair} if $V_i'\subset V_i$ and $V_j'\subset V_j$ such that both $|V_i'|\geq \gamma |V_i|$ and $|V_j'|\geq \gamma |V_j|$, then $|d_{\rho}(V_i',V_j')-d_{\rho}(V_i,V_j)|\leq\gamma$ for each $\rho\in\{1,\ldots,r\}$.  Equivalently, $\|\bfd(V_i',V_j')-\bfd(V_i,V_j)\|_{\infty}\leq\gamma$.

A partition $(V_1,\ldots,V_k)$ of the vertex set of $G$, an $r$-colored graph on $n$ vertices, is said to be a \textdef{$\gamma$-regular partition} if each of the following holds:
\begin{itemize}
   \item $\left||V_i|-|V_j|\right|\leq 1$ for all $i,j\in\{1,\ldots,k\}$.
   \item All but at most $\gamma k^2$ of the pairs $(V_i,V_j)$, $1\leq i<j\leq k$ are $\gamma$-regular.
\end{itemize}

The multicolor version of Szemer\'edi's lemma can be easily derived from a proof outline by Koml\'os and Simonovits~\cite{KS}:
\begin{thm}[Szemer\'edi~\cite{Sz}]\label{thm:reglem:multicol}
   Fix an integer $r\geq 2$.  For every $\epsilon>0$, and positive integer $m$, there exists an integer $CM=CM(m,\epsilon)$ with the following property.

   If $G$ is an $r$-graph with $n\geq CM$ vertices, and $\A$ is an equipartition of the vertex set of $G$ with an order not exceeding $m$, then there exists a refinement $\B$ of $\A$ of order $k$, where $m\leq k\leq CM$ which is $\epsilon$-regular.
\end{thm}

The classical formulation of Szemer\'edi's regularity lemma provides only the existence of the $\epsilon$-regular partition.  However, its proof implies the more precise refinement result we state as Theorem~\ref{thm:reglem:multicol}. In addition, the classical formulation of the lemma allows for an exceptional set of size at most $\epsilon n$.  We can, however, apply the original formulation to the graph $G$ with a smaller parameter than $\epsilon$ and evenly distribute the vertices in the exceptional set among the other clusters to get the result with the given value of $\epsilon$.

Multicolored graphs have their own Slicing and Embedding lemmas:
\begin{lem}[Slicing lemma]\label{lem:slicing:multicol}
   If $(A,B)$ is a $\gamma$-regular pair in an $r$-graph such that $(A,B)$ has density vector $(d_1,\ldots,d_r)$ and $A'\subset A$ and $B'\subset B$ satisfy $|A'|\geq\epsilon |A|$ and $|B'|\geq\epsilon |B|$ for some $\epsilon\geq\gamma$, then $(A',B')$ is a $(\max\{2,\epsilon^{-1}\}\gamma)$-regular pair with density vector $\bfd=(A',B')$ such that $|d_{\rho}(A,B)-d_{\rho}(A',B')|\leq\gamma$ for each $\rho\in\{1,\ldots,r\}$ (equivalently, $\|\bfd(A,B)-\bfd(A',B')\|_{\infty}\leq\gamma$).
\end{lem}

\begin{proof}
Let $\eta=\max\{2,\epsilon^{-1}\}\gamma$.  We may assume $\eta<1$, otherwise the lemma is trivially true as all pairs are $\eta$-regular whenever $\eta\geq 1$.  In order to verify the regularity of $(A',B')$, choose $A''\subset A'$ and $B''\subset B'$ such that $|A''|\geq\eta|A'|$ and $|B''|\geq\eta|B'|$.  Consequently,
$$ |A''|\geq\eta|A'|\geq\eta\epsilon|A|=\max\{2\epsilon,1\}\gamma|A|\geq\gamma|A| $$
and similarly, $|B''|\geq\gamma|B|$.  By the $\gamma$-regularity of $(A,B)$, we know that $(A'',B'')$ has density vector $\bfd(A'',B'')$ such that $\|\bfd(A,B)-\bfd(A'',B'')\|_{\infty}\leq\gamma$.  Moreover, since $|A'|\geq|A''|\geq\gamma|A|$ and $|B'|\geq|B''|\geq\gamma|B|$, then  $\|\bfd(A,B)-\bfd(A',B')\|_{\infty}\leq\gamma$.  By the triangle inequality,
$$ \|\bfd(A',B')-\bfd(A'',B'')\|_{\infty}\leq \|\bfd(A,B)-\bfd(A',B')\|_{\infty}+\|\bfd(A,B)-\bfd(A'',B'')\|_{\infty}\leq 2\gamma\leq\eta . $$
The arbitrary choice of $A''$ and $B''$ means that $(A',B')$ is $\eta$-regular.
\end{proof}

\newcommand{\gamembedmulticol}{\gamma_{\ref{lem:embed:multicol}}}
\newcommand{\delembedmulticol}{\delta_{\ref{lem:embed:multicol}}}
\begin{lem}[Embedding lemma]\label{lem:embed:multicol}
   For every $0<\eta<1$ and positive integer $k$ there exist $\gamma=\gamembedmulticol(\eta,k)$ and $\delta=\delembedmulticol(\eta,k)$ with the following property. \\
   \indent Fix an integer $r\geq 2$. Suppose that $H=(\{v_1,\ldots,v_k\},c)$ is an $r$-graph. Let $G$ be an $r$-graph.  Let $V_1,\ldots,V_k$ be a $k$-tuple of disjoint vertex sets of $G$ such that for every $1\leq i<i'\leq k$ the pair $(V_i,V_{i'})$ is $\gamma$-regular, such that the density $d_{\rho}(V_i,V_{i'})\geq\eta$ if $v_iv_{i'}$ is an edge of $H$ with color $\rho$, for each $\rho$, $1\leq\rho\leq k$.  Then, at least $\delta\prod_{i=1}^k|V_i|$ of the $k$-tuples $(w_1,\ldots,w_k)$ with $w_1\in V_1,\ldots,w_k\in V_k$ span copies of $H$ where each $w_i$ plays the role of $v_i$.

   Note that the case of $r=2$ is the case of induced graphs in which edges are color 1 and nonedges are color 2.
\end{lem}

\begin{proof}
We note that $r$ plays no role at all in the definitions of $\gamma$ and $\delta$.  This is because $\eta$ is the parameter that ensures the proper density for all colors.  We will choose $\gamembedmulticol(\eta,k)=\min\left\{(\eta/2)^{k-1},(1/6)^{k-1}\right\}$.

We proceed via induction on $k$ to determine the value of $\delembedmulticol(\eta,k)$.  The case of $k=1$ is trivial and $\delembedmulticol(\eta,1)=1$ for all $\eta$.  Let $k\geq 2$ and suppose there is such a function $\delembedmulticol(\eta,k-1)$.  Let
\begin{equation}\label{eq:gamma}
   \gamma=\min\left\{(\eta/2)^{k-1},(1/6)^{k-1}\right\} .
\end{equation}

Consider $V_k$.  Call a vertex $w_k\in V_k$ bad if, for some $i\in\{1,\ldots,k-1\}$, $w_k$ has less than $(\eta-\gamma)|V_i|$ edges of color $\rho=c(v_k,v_i)$ incident to it with the other endpoint in $V_i$.

Assume that more than $\gamma|V_k|$ vertices in $V_k$ are bad and let $V_k'$ be the set of bad vertices.  Then,
$d_{\rho}(V_k',V_i)<\frac{(\eta-\gamma)|V_k'||V_i|}{|V_k'||V_i|}=\eta-\gamma$. On the other hand $d_\rho(V_k,V_i)\geq\eta$.  So $|d_{\rho}(V_k',V_i)-d_{\rho}(V_k,V_i)|>\gamma$, contradicting the fact that $(V_k,V_i)$ is $\gamma$-regular.

Thus, the number of bad vertices is at most $\gamma|V_k|$.  Therefore, thare are at most $(k-1)\gamma|V_k|<|V_k|$ vertices that are bad with respect to some $V_i$, $1\leq i\leq k-1$. Let $w_k\in V_k$ be a vertex that is not bad with respect to each $V_i$. Let $\overline{V}_i\subset V_i$ be a set of $\left\lceil (\eta-\gamma)|V_i|\right\rceil$ vertices $w_i$ such that $w_iw_k$ has the correct color; i.e., the color of $v_iv_k$.

By the Slicing Lemma, each pair $(\overline{V}_i,\overline{V}_{i'})$ for $1\leq i<i'\leq k-1$ is $\left(\max\left\{2,(\eta-\gamma)^{-1}\right\}\gamma\right)$-regular.  The pairs also have that $d_{\rho}(V_i,V_{i'})\geq\eta-\gamma$ if $v_iv_{i'}$ is an edge of $H$ with color $\rho$.

In order to apply the inductive hypothesis, we must verify that
\begin{equation}\label{eq:gamcheck}
\max\left\{2,(\eta-\gamma)^{-1}\right\}\gamma \leq\gamembedmulticol(\eta-\gamma,k-1) =\min\left\{\left(\frac{\eta-\gamma}{2}\right)^{k-2}, \left(\frac{1}{6}\right)^{k-2}\right\} .
\end{equation}

If $\eta-\gamma\geq 1/2$, then (\ref{eq:gamma}) gives that $\gamma=(1/6)^{k-1}$ and (\ref{eq:gamcheck}) reduces to $2\gamma\leq (1/6)^{k-2}$, which is true for all $k$.

If $\eta-\gamma<1/2$ and $\eta-\gamma\geq 1/3$, then (\ref{eq:gamma}) gives that $\gamma\leq(1/6)^{k-1}$ and (\ref{eq:gamcheck}) reduces to
$$ \frac{\gamma}{\eta-\gamma}\leq\left(\frac{1}{6}\right)^{k-2} . $$
This is true because $\gamma/(\eta-\gamma)\leq 3\gamma=3 (1/6)^{k-1}=\frac{1}{2}(1/6)^{k-2}$.

If $\eta-\gamma<1/3$, since $\gamma\leq (\eta/2)^{k-1}$, (\ref{eq:gamcheck}) reduces to
\begin{eqnarray*}
   \frac{\gamma}{\eta-\gamma} & \leq & \left(\frac{\eta-\gamma}{2}\right)^{k-2} \\
   2^{k-2}\gamma & \leq & (\eta-\gamma)^{k-1} .
\end{eqnarray*}
To verify this, see that
$$ 2^{k-2}\gamma\leq 2^{k-2}(\eta/2)^{k-1}=\frac{1}{2}\eta^{k-1} $$
and that
$$ (\eta-\gamma)^{k-1} \geq\left(\eta-(\eta/2)^{k-1}\right)^{k-1} =\eta^{k-1}\left(1-\frac{\eta^{k-2}}{2^{k-1}}\right)^{k-1} \geq\eta^{k-1}\left(1-2^{1-k}\right)^{k-1} . $$
Some calculus shows that $(1-2^{-x})^x$ is increasing for $x\geq 1$ and so we have
$$ (\eta-\gamma)^{k-1}\geq\eta^{k-1}\left(1-2^{1-2}\right)^{2-1}=\frac{1}{2}\eta^{k-1} , $$
as needed.

Now that we have verified that we can use the inductive hypothesis, we do so and see that the number of copies of $H-v_k$ in $\left(\overline{V}_1,\ldots,\overline{V}_{k-1}\right)$ at least $\delembedmulticol(\eta-\gamma,k-1)\prod_{i=1}^{k-1}\left|\overline{V}_i\right|$.  So the total number of copies of $H$ is at least
\begin{eqnarray*}
\lefteqn{\delembedmulticol(\eta-\gamma,k-1)\prod_{i=1}^{k-1}\left|\overline{V}_i\right|\cdot\left(1-(k-1)\gamma\right)|V_k|} \\ & \geq & \delembedmulticol(\eta-\gamma,k-1)(\eta-\gamma)^{k-1}\left(1-(k-1)\gamma\right)\prod_{i=1}^k|V_i| . \end{eqnarray*}

With $\gamma=\min\left\{(\eta/2)^{k-1},(1/6)^{k-1}\right\}$, set $\delembedmulticol(\eta,k)=\delembedmulticol(\eta-\gamma,k-1)(\eta-\gamma)^{k-1}\left(1-(k-1)\gamma\right)$, the conditions of the Embedding Lemma are satisfied.
\end{proof}

\subsection{Directed graph version}
A \textdef{digraph} is defined to be a pair $(V,E)$ where $V$ is a labeled vertex set, $E\subseteq (V)_2$ and $(V)_2$ denotes the set $V\times V-\{(v,v) : v\in V\}$. It is convenient for us to view this as a coloring. That is, a digraph is a pair $(V,c)$ where $c:(V)_2\rightarrow\{\nonarrow,\unarrow,\leftarrow,\rightarrow\}$ is a function known as the \textdef{partial orientation} of the edge set.  It has the property that, for distinct $v,w$,
\begin{itemize}
   \item $c(v,w)=c(w,v)$ if and only if $c(v,w)\in\{\nonarrow,\unarrow\}$ and
   \item $c(v,w)=\rightarrow$ if and only if $c(w,v)=\leftarrow$.
\end{itemize}
For convenience, we denote $\Aset:=\{\nonarrow,\unarrow,\leftarrow,\rightarrow\}$.  Here we interpret the color $c(v,w)=\nonarrow$ to mean that neither $(v,w)$ nor $(w,v)$ are in $E$, the color $c(v,w)=\unarrow$ to mean that both $(v,w)$ and $(w,v)$ are in $E$ and the color $c(v,w)=\rightarrow$ to mean that $(v,w)\in E$ and $(w,v)\not\in E$.

In the directed case, we have the same notions of $\gamma$-regular pairs as in the multicolor case.  The \textdef{density vector} of the pair $(V_i,V_j)$ is somewhat similar as well:
$$ \bfd(V_i,V_j):=\left(d_{\nonarrow}(V_i,V_j),d_{\unarrow}(V_i,V_j),d_{\leftarrow}(V_i,V_j),d_{\rightarrow}(V_i,V_j)\right) . $$
However, in the directed case, the order makes a difference.  Although $d_{\rho}(A,B)=d_{\rho}(B,A)$ for $\rho\in\{\nonarrow,\unarrow\}$, it is also the case that $d_{\rightarrow}(A,B)=d_{\leftarrow}(B,A)$.




Alon and Shapira give the following version of Szemer\'edi's lemma:
\label{sec:digraph}
\begin{thm}[Alon-Shapira~\cite{ASh}]\label{thm:reglem:digraph}
   For every $\epsilon>0$ and positive integer $m$, there exists an integer $DM=DM(m,\epsilon)$ with the following property.

   If $G$ is a digraph $n\geq DM$ vertices, and $\A$ is an equipartition of the vertex set of $G$ with an order not exceeding $m$, then there exists a refinement $\B$ of $\A$ of order $k$, where $m\leq k\leq DM$ which is $\epsilon$-regular.
\end{thm}


Digraphs have their own Slicing and Embedding lemmas:
\begin{lem}[Slicing lemma]\label{lem:slicing:digraph}
   If $(A,B)$ is a $\gamma$-regular pair in a digraph such that $(A,B)$ has density vector $(d_{\nonarrow},d_{\unarrow},d_{\leftarrow},d_{\rightarrow})$ and $A'\subset A$ and $B'\subset B$ satisfy $|A'|\geq\epsilon |A|$ and $|B'|\geq\epsilon |B|$ for some $\epsilon\geq\gamma$, then $(A',B')$ is a $(\max\{2,\epsilon^{-1}\}\gamma)$-regular pair with density vector $\bfd':=(d_{\nonarrow}',d_{\unarrow}',d_{\leftarrow}',d_{\rightarrow}')$ such that $|d_{\rho}-d'_{\rho}|\leq\gamma$ for each $\rho\in\{\nonarrow,\unarrow,\leftarrow,\rightarrow\}$ (equivalently $\|\bfd(A,B)-\bfd(A',B')\|_{\infty}\leq\gamma$).
\end{lem}

The proof is identical to the multicolor case, Lemma~\ref{lem:slicing:multicol}.

\begin{lem}[Embedding lemma]\label{lem:embed:digraph}
\newcommand{\gamembeddigraph}{\gamma_{\ref{lem:embed:digraph}}}
\newcommand{\delembeddigraph}{\delta_{\ref{lem:embed:digraph}}}
   For every $0<\eta<1$ and positive integer $k$ there exist $\gamma=\gamembeddigraph(\eta,k)$ and $\delta=\delembeddigraph(\eta,k)$ with the following property. \\
   \indent Suppose that $H$ is a digraph with vertices $v_1,\ldots,v_k$, and that $V_1,\ldots,V_k$ is a $k$-tuple of disjoint vertex sets of $G$ such that for every $1\leq i<i'\leq k$ the pair $(V_i,V_{i'})$ is $\gamma$-regular, such that the density $d_{\rho}(V_i,V_{i'})\geq\eta$ if $(v_i,v_{i'})$ is an edge of $H$ with color $\rho$.  Then, at least $\delta\prod_{i=1}^k|V_i|$ of the $k$-tuples $(w_1,\ldots,w_k)$ with $w_1\in V_1,\ldots,w_k\in V_k$ span (induced) copies of $H$ where each $w_i$ plays the role of $v_i$.
\end{lem}

Again, the proof is identical to the multicolor case, Lemma~\ref{lem:slicing:multicol}.

\subsection{Main results}
\label{sec:main}
The statement of the main result (Theorem~\ref{thm:main}) can be made in general with the definitions above.  Recall that $\bfd(V_i,V_{i'})$ denotes the density vector of the pair $(V_i,V_{i'})$.
\begin{thm}[Alon, et al.~\cite{AFKSz}]\label{thm:main}
\newcommand{\Smain}{S_{\ref{thm:main}}}
\newcommand{\delmain}{\delta_{\ref{thm:main}}}
   Fix $r\geq 2$.  For every $m$ and function $\E$ with $\E:\N\rightarrow (0,1)$, there exist $S=\Smain(r,m,\E)$ and $\delta=\delmain(r,m,\E)$ with the following property: \\
   \indent If $G$ is a graph [$r$-graph, digraph] with $n\geq S$ vertices then there exist an equipartition $\A=\{V_i : 1\leq i\leq k\}$ of $G$ and an induced subgraph [induced $r$-subgraph, induced subdigraph] $G'$ of $G$, with an equipartition $\A'=\{V_i' : 1\leq i\leq k\}$ of the vertices of $G'$ that satisfy:
   \begin{itemize}
      \item $S\geq k\geq m$.
      \item $V_i'\subset V_i$ for all $i\geq 1$, and $|V_i'|\geq\delta n$.
      \item In the equipartition $\A'$, \textbf{all} pairs are $\E(k)$-regular.
      \item All but at most $\E(0)\binom{k}{2}$ of the pairs $1\leq i<i'\leq k$ are such that $\|\bfd(V_i,V_{i'})-\bfd(V_i',V_{i'}')\|_{\infty}<\E(0)$.
   \end{itemize}
\end{thm}

Our contribution is to prove the case for multicolored graphs and digraphs.  Although the proof is quite similar to that of N. Alon, E. Fischer, M. Krivelevich and M. Szegedy~\cite{AFKSz}, there are subtleties that need to be addressed.

\section{Proof of the main results}
\label{sec:proofs}
There is a plethora of lemmas that are required to prove our main result.  Lemma~\ref{lem:defectCS} is a consequence of the defect form of the Cauchy-Schwarz Inequality which is stated without proof and can be found in~\cite{Sz}.  Corollary~\ref{cor:defectCS} is a direct consequence of Lemma~\ref{lem:defectCS}.  Lemma~\ref{lem:AFKSz:refine} is a refinement lemma that allows the induction to take place and Lemma~\ref{lem:main} is the main lemma, of which our main result, Theorem~\ref{thm:main} is a direct consequence.

First, we need a definition which, in the context of multicolorings of the complete graph, comes from \cite{KS}.
\newcommand{\ind}{{\rm ind}}
\begin{defn}
Given an equipartition $\A=\{V_i : 1\leq i\leq k\}$ of the vertex set of a multicolored graph [digraph], we define the \textdef{index of $\A$} as follows:
$$ \ind(\A)=\frac{1}{k^2}\sum_{\rho}\sum_{1\leq i<i'\leq k}d_{\rho}^2(V_i,V_{i'}) , $$
\end{defn}
where in the case of multicolored graphs, $\rho$ runs over all colors and in the case of digraphs, the colors $\rho$ run over the set of four ``colors'' in the set $\Aset=\{\nonarrow,\unarrow,\leftarrow,\rightarrow\}$.

Note also that $\ind(\A)=\frac{1}{k^2}\sum_{1\leq i<i'\leq k}\sum_{\rho}d_{\rho}^2(V_i,V_{i'})\leq \frac{1}{k^2}\sum_{1\leq i<i'\leq k}\left(\sum_{\rho}d_{\rho}(V_i,V_{i'})\right)^2\leq\frac{1}{2}$.

\begin{lem}\label{lem:defectCS}
For all sequences of nonnegative numbers $X_1,\ldots,X_n$, if for some $m$, $1\leq m<n$
$$ \sum_{k=1}^mX_k=\frac{m}{n}\sum_{k=1}^nX_k+\alpha , $$
then
$$ \sum_{k=1}^nX_k^2\geq\frac{1}{n}\left(\sum_{k=1}^nX_k\right)^2+\frac{\alpha^2n}{m(n-m)} . $$
(observe that $\alpha$ need not be positive).
\end{lem}

\begin{cor}\label{cor:defectCS}
Suppose that $A$ and $B$ are two disjoint sets of vertices of a multicolored graph [digraph] $G$, and $\{A_j : 1\leq j\leq\ell\}$ and $\{B_j : 1\leq j\leq\ell\}$ are their two respective partitions to sets of \textbf{equal} sizes, such that, for some color $\rho$, at least $\epsilon\ell^2$ of the possible $j,j'$ satisfy $|d_{\rho}(A,B)-d_{\rho}(A_j,B_{j'})|\geq\frac{1}{2}\epsilon$.  Then,
$$ \sum_{1\leq j,j'\leq\ell}d^2_{\rho}(A_j,B_{j'})>\ell^2\left(d^2_{\rho}(A,B)+\frac{1}{8}\epsilon^3\right) . $$
\end{cor}

\begin{proofcite}{Corollary~\ref{cor:defectCS}}
Under the above conditions, either at least $\frac{1}{2}\epsilon\ell^2$ of the pairs $j,j'$ are such that $d_{\rho}(A_j,B_{j'})-d_{\rho}(A,B)\geq\frac{1}{2}\epsilon$, or at least $\frac{1}{2}\epsilon\ell^2$ are such that $d_{\rho}(A_j,B_{j'})-d_{\rho}(A,B)\leq -\frac{1}{2}\epsilon$.  We use Lemma~\ref{lem:defectCS} with $n=\ell^2$, $m=\frac{1}{2}\epsilon\ell^2$, and $\alpha$ satisfying $|\alpha|\geq\frac{1}{4}\epsilon^2\ell^2$.  Furthermore, we use the fact that all $|A_j|=|A|/\ell$ and all $|B_{j'}|=|B|/\ell$ to obtain
$$ \sum_{1\leq j,j'\leq\ell}d_{\rho}(A_j,B_{j'})=\ell^2 d(A,B) . $$

Applying Lemma~\ref{lem:defectCS} to the sequence $\left\{d_{\rho}(A_j,B_{j'})\right\}_{1\leq j,j'\leq\ell}$, we obtain
$$ \sum_{1\leq j,j'\leq\ell}d_{\rho}^2(A_j,B_{j'}) \geq \ell^2d_{\rho}^2(A,B) +\frac{\frac{1}{16}\epsilon^4\ell^6}{\frac{1}{2}\epsilon\ell^2(\ell^2-\frac{1}{2}\epsilon\ell^2)} > \ell^2\left(d_{\rho}^2(A,B)+\frac{1}{8}\epsilon^3\right) $$
as required.
\end{proofcite}

\begin{lem}\label{lem:AFKSz:refine}
Suppose that $\A=\{V_i : 1\leq i\leq k\}$ and its refinement $\B=\{V_{i,j} : 1\leq i\leq k, 1\leq j\leq\ell\}$ be vertex partitions of a graph $G$, satisfying $\ind(\B)-\ind(\A)\leq\frac{1}{64}r\epsilon^4$ for some $\epsilon$, and that the number of vertices of the graph is $n>512\epsilon^{-4}rk\ell$.  Then, for all possible $i<i'$ but at most $\epsilon\binom{k}{2}$ of them, $|d_{\rho}(V_i,V_{i'})-d_{\rho}(V_{i,j},V_{i',j'})|<\epsilon$ holds simultaneously for all $\rho$, for all but a maximum of $\epsilon\ell^2$ of the possible $j,j'$.
\end{lem}

\begin{proofcite}{Lemma~\ref{lem:AFKSz:refine}}
Supposing the contrary and assuming $\epsilon<1$ and $k>1$, we show that the index of $\B$ is larger than that of $\A$ by more than $\frac{1}{64}r\epsilon^4$.  If not all of the sets of $\B$ are of exactly the same size, let $V_{i,j}'$ be $V_{i,j}$ for sets of the smaller size and $V_{i,j}'$ be $V_{i,j}$ minus an arbitrarily chosen vertex for sets of the larger size.  Defining also $V_i'=\bigcup_{1\leq j\leq\ell}V_{i,j}'$, we define two new partitions $\B'=\{V_{i,j}' : 1\leq i\leq k, 1\leq j\leq\ell\}$ and $\A'=\{V_i' : 1\leq i\leq k\}$ of a large induced submulticolored graph [subdigraph] of $G$ (for each of these new partitions all its sets are of the same size).  The assumption on $n$ implies that $|d_{\rho}(V_i,V_{i'})-d_{\rho}(V_i',V_{i'}')|<\frac{1}{256}\epsilon^4$ and $|d_{\rho}(V_{i,j},V_{i',j'})-d_{\rho}(V_{i,j}',V_{i',j'}')| < \frac{1}{256}\epsilon^4$ hold for all $i,j,i',j',\rho$.  In particular, $|\ind(\A)-\ind(\A')|<\frac{1}{128}\epsilon^4$ and $|\ind(\B)-\ind(\B')|<\frac{1}{128}\epsilon^4$ hold, and for more than $\epsilon\binom{k}{2}$ of the possible $i<i'$, the inequality $|d_{\rho}(V_i',V_{i'}')-d_{\rho}(V_{i,j}',V_{i',j'}')| > \epsilon-\frac{2}{256}\epsilon^4>\frac{1}{2}\epsilon$ holds for at least $\epsilon\ell^2$ of the possible $j,j'$.  Using Corollary~\ref{cor:defectCS}, we obtain
\begin{eqnarray*}
   \ind(\B') & \geq & \frac{1}{k^2\ell^2} \sum_{\rho}\sum_{\scriptsize\begin{array}{c} 1\leq i<i'\leq k \\ 1\leq j,j'\leq\ell \end{array}}d_{\rho}^2(V_{i,j}',V_{i',j'}') \\
   & > & \frac{1}{k^2\ell^2}\sum_{\rho}\left(\ell^2\sum_{1\leq i<i'\leq k}d_{\rho}^2(V_{i}',V_{i'}')+\epsilon\binom{k}{2}\ell^2\frac{1}{8}\epsilon^3\right) \geq \ind(\A')+\frac{1}{32}r\epsilon^4 .
\end{eqnarray*}
This implies $\ind(\B)-\ind(\A)\geq\ind(\B')-\ind(\A')-\frac{2}{128}r\epsilon^4>\frac{1}{64}r\epsilon^4$, completing the proof.
\end{proofcite}

The main lemma is Lemma~\ref{lem:main}.
\newcommand{\Smain}{S_{\ref{lem:main}}}
\begin{lem}\label{lem:main}
Fix a positive integer $r$.  For every integer $m$ and function $\E$ with $\E:\N\rightarrow (0,1)$, there exists a number $S=\Smain(r,m,\E)$ with the following property.

If $G$ is an $r$-graph [digraph] with $n\geq S$ vertices, then there exists an equipartition $\A=\{V_i : 1\leq i\leq k\}$ and a refinement $\B=\{V_{i,j} : 1\leq i\leq k, 1\leq j\leq\ell\}$ of $\A$ that satisfy:
\begin{itemize}
   \item $|\A|=k\geq m$ but $|\B|=k\ell\leq S$.
   \item For all $1\leq i<i'\leq k$ but at most $\E(0)\binom{k}{2}$ of them, the pair $(V_i,V_{i'})$ is $\E(0)$-regular.
   \item For all $1\leq i<i'\leq k$ and all $1\leq j,j'\leq\ell$ but at most $\E(k)\ell^2$ of them, the pair $(V_{i,j},V_{i',j'})$ is $\E(k)$-regular.
   \item All $1\leq i<i'\leq k$ but at most $\E(0)\binom{k}{2}$ of them are such that for all $1\leq j,j'\leq\ell$ but at most $\E(0)\ell^2$ of them $|d_{\rho}(V_i,V_{i'})-d_{\rho}(V_{i,j},V_{i',j'})|<\E(0)$ holds for each $\rho\in\{1,\ldots,r\}$.
\end{itemize}
\end{lem}

\begin{proof}
We may assume that $m>1$ and that $\E(\kappa)$ is monotone nonincreasing.  For convenience, let $\epsilon=\E(0)$.

If we are in the case of a multicolored graph, fix a positive integer $r$, and using the function $CM$ from Theorem~\ref{thm:reglem:multicol}, let
$$ T^{(1)}=CM(r,m,\epsilon) $$
and for $i>1$, we define by induction
$$ T^{(i)}=CM(r,T^{(i-1)},2\E(T^{(i-1)})(T^{(i-1)})^{-2}) . $$

If we are in the case of a digraph, and using the function $DM$ from Theorem~\ref{thm:reglem:digraph}, let
$$ T^{(1)}=DM(m,\epsilon) $$
and for $i>1$, we define by induction
$$ T^{(i)}=DM(T^{(i-1)},2\E(T^{(i-1)})(T^{(i-1)})^{-2}) . $$

In either case, we show that $S=512r\epsilon^{-4}T^{(64r\epsilon^{-4}+1)}$ satisfies the required property.

Given $G$, define $\A_1$ to be an equipartition of order at least $m$ but not greater than $T^{(1)}$, such that all pairs but at most $\epsilon\binom{|\A_1|}{2}$ of them are $\epsilon$-regular.  Define by induction for $i>1$ the equipartition $\A_i$ to be a refinement of $\A_{i-1}$, of order not greater than $T^{(i)}$ such that all of the pairs but at most
$$ 2\E(T^{(i-1)})(T^{(i-1)})^{-2}\binom{|\A_i|}{2}\leq 2\E(T^{(i-1)})(|\A_{i-1}|)^{-2}\binom{|\A_i|}{2} $$
are $2\E(T^{(i-1)})(T^{(i-1)})^{-2}<\E(T^{(i-1)})$-regular.  The refinements are guaranteed by the original regularity lemma, either Theorem~\ref{thm:reglem:multicol} (in the multicolor case) or Theorem~\ref{thm:reglem:digraph} (in the digraph case).

Let us now choose the minimum $i$ such that $\ind(\A_i)-\ind(\A_{i-1})\leq\frac{1}{64}r\epsilon^4$.  There certainly exists such an $1<i\leq 64r^{-1}\epsilon^{-4}+1$ since the indices of each partition in the series are all between $0$ and $1$.  We set $\A=\A_{i-1}$ and $\B=\A_i$, and appropriately $k=|\A_{i-1}|=|\A|$ and $l=k^{-1}|\A_i|=|\A|^{-1}|\B|$.  We claim that $\A$ and $\B$ are the required partitions.

It is clear that $\B$ is a refinement of $\A$ and that they both satisfy the requirements with regards to their respective orders.  It is also clear (by the assumption $\E(\kappa)\leq\E(0)=\epsilon$) that $\A$ satisfies the requirement regarding the regularity of its pairs.  Since all but at most $2\E(k)k^{-2}\binom{k\ell}{2}<\E(k)\ell^2$ of all the pairs of $\B$ are $\E(k)$-regular, the condition regarding the regularity of pairs of $\B$ in the formulation of the lemma follows.  Finally, Lemma~\ref{lem:AFKSz:refine} shows that most densities of the pairs of $\B$ differ from the corresponding densities of the pairs of $\A$ by less than $\epsilon$, as in the formulation of the last condition of this lemma.
\end{proof}

\begin{proofcite}{Theorem~\ref{thm:main}}
We may assume $\E(\kappa)\leq\E(0)$.  Set $\epsilon=\E(0)$.  Define $\E'$ by setting $\E'(\kappa)\linebreak=\min\left\{\E(\kappa),\frac{1}{4}\epsilon,\frac{1}{2}\binom{\kappa+2}{2}^{-1}\right\}$, set $S=\Smain(r,m,\E')$ and $\delta=\frac{1}{2}(\Smain(m,\E'))^{-1}$. Use Lemma~\ref{lem:main} on $G$, finding the appropriate partitions $\A=\{V_i : 1\leq i\leq k\}$ and $\B=\{V_{i,j} : 1\leq i\leq k, 1\leq j\leq\ell\}$.

Now choose randomly, independently and uniformly $j_i$ such that $1\leq j_i\leq\ell$ for each $1\leq i\leq k$.  With probability more than $1/2$, all the pairs $\left(V_{i,j_i},V_{i',j_{i'}}\right)$ are $\E'(k)$-regular.  In fact, the probability that there is some pair that is not $\E(k)$-regular is at most $\E(k)'\binom{k}{2}$.

Moreover, the expected number of pairs $1\leq i\leq i'\leq k$ for which $|d_{\rho}(V_i,V_{i'})-d_{\rho}(V_{i,j_i},V_{i',j_{i'}})|\geq\epsilon$ for some $\rho$ is no more than $\frac{1}{4}\epsilon\binom{k}{2}+\frac{1}{4}\epsilon\binom{k}{2}=\frac{1}{2}\epsilon\binom{k}{2}$, by the choice of $\E'$, so with probability at least $1/2$, no more than $\epsilon\binom{k}{2}$ of the pairs satisfy this.

Therefore, there exists a choice of $j_1,\ldots,j_k$ such that all pairs $\left(V_{i,j_i},V_{i',j_{i'}}\right)$ are $\E(k)$-regular, and all but at most $\epsilon\binom{k}{2}$ of them satisfy $|d_{\rho}(V_i,V_{i'})-d_{\rho}(V_i',V_{i'}')|<\epsilon$ for all $\rho\in\{1,\ldots,r\}$. Defining $G'$ as the induced subgraph spanned by $\bigcup_{1\leq i\leq k}V_{i,j_i}$, and $\A'$ by setting $V_i=V_{i,j_i}$ achieves the required result.
\end{proofcite}

\section{Application}
\label{sec:apps}
An important feature of editing is the notion of the palette.  Colloquially, the \textdef{palette} is the set of colors to which an edge can be changed. For an $r$-graph, the palette is always the set $\{1,\ldots,r\}$.  Note that if $r=2$, this is the case of simple graphs. So, we will not define the palette for $r$-graphs, only focusing on it for digraphs.

\begin{defn}
In the case of digraphs, we say that $\Pset\subseteq\Aset$ is a \textdef{palette} if either none or both of ``$\rightarrow$'' and ``$\leftarrow$'' are in $\Pset$ and every digraph is a pair $(V,c)$ where $V$ is a vertex set and $c:(V)_2\rightarrow\Pset$ is a coloring of the edge set of a complete graph on $|V|$ vertices. There are 5 possible nontrivial palettes:
\begin{enumerate}
\setcounter{enumi}{-1}
\item $\Pset_0=\Aset$ is the most general case. \label{it:palette:A}
\item $\Pset_1=\{\unarrow,\leftarrow,\rightarrow\}$ is the case of simple digraphs such that every pair of vertices has at least one arc between them. \label{it:palette:noempty}
\item $\Pset_2=\{\nonarrow,\leftarrow,\rightarrow\}$ is the case of \textdef{oriented graphs}; that is, no pair of vertices has two arcs between them. \label{it:palette:nounarrow}
\item $\Pset_3=\{\nonarrow,\unarrow\}$ is the case of simple, undirected graphs.  \label{it:palette:undir}
\item $\Pset_4=\{\leftarrow,\rightarrow\}$ is the case of \textdef{tournaments}. \label{it:palette:tourn}
\end{enumerate}
\end{defn}

Recall that the vector is of the form $(p,q)$ where $p,q\geq 0$ and $0\leq 1-p-2q$.  In the cases in which the palette is not $\Aset$, the relevant density vector must be further restricted.
\begin{itemize}
   \item[(\ref{it:palette:noempty})] In the case of $\Pset_1=\{\unarrow,\leftarrow,\rightarrow\}$, then $p+2q=1$.
   \item[(\ref{it:palette:nounarrow})] In the case of $\Pset_2=\{\nonarrow,\leftarrow,\rightarrow\}$, then $p=0$ and $q\leq 1/2$.
   \item[(\ref{it:palette:undir})] In the case of $\Pset_3=\{\nonarrow,\unarrow\}$, then $q=0$ and $p\leq 1$. This is the $r$-graph case where $r=2$ or simply the case of undirected graphs.  See~\cite{AS} and~\cite{AKM}.
   \item[(\ref{it:palette:tourn})] In the case of $\Pset_4=\{\leftarrow,\rightarrow\}$, then $p=0$ and $1-p-2q=0$, so $q=1/2$.
\end{itemize}

Our application is one of edit distance and it shows that $r$-types [dir-types] are used to lower bound the edit distance function. It turns out that, trivially, they upper bound the edit distance function.

\begin{defn}
An \textdef{$r$-type}, $K$, is a pair $(U,\phi)$, where $U$ is a finite set of vertices and $\phi : U\times U \rightarrow 2^{\{1,\ldots,r\}}\setminus \emptyset$, such that $\phi(x,y)=\phi(y,x)$ and $\phi(x,x)\neq\{1,\ldots,r\}$, for all $x, y\in U$. Informally, we will view an $r$-type as a complete graph with a coloring of both vertices and edges using subsets of $\{1,\ldots,r\}$. The \textdef{sub-$r$-type} of $K$ induced by $W\subseteq U$ is the $r$-type achieved by deleting the vertices $U-W$ from $K$.

We say that an $r$-graph $H=(V,c)$ of a complete graph \textdef{embeds in type $K=(U,\phi)$}, and write $H\mapsto K$, if there is a map $\gamma:V\rightarrow U$ such that $c(\{v,v'\})=c_0$ implies $c_0\in\phi(\gamma(v),\gamma(v'))$.
\end{defn}

Types are defined in a slightly different way for digraphs.
\begin{defn}
Let $\Pset\subseteq\Aset$ be a palette. A \textdef{$\Pset$-dir-type} or simply \textdef{dir-type} where the palette is understood, $K$, is a pair $(U,\phi)$, where $U$ is a finite set of vertices and $\phi : U\times U \rightarrow 2^{\Pset}\setminus \emptyset$, such that
\begin{enumerate}
\item for distinct $x,y$ and $\rho\in\{\nonarrow,\unarrow\}$, $\phi(x,y)\ni\rho$ if and only if $\phi(y,x)\ni\rho$ and
\item $\phi(x,y)\ni\rightarrow$ if and only if $\phi(y,x)\ni\leftarrow$.
\end{enumerate}
Moreover, for all $x\in U$, $\phi(x,x)$ is a nonempty proper subset of $\Pset$. The \textdef{sub-dir-type} of $K$ induced by $W\subseteq U$ is the dir-type achieved by deleting the vertices $U-W$ from $K$.

We say that a directed graph $H=(V,c)$ \textdef{embeds in type $K=(U,\phi)$}, and write $H\mapsto K$, if there is a map $\gamma:V\rightarrow U$ such that, for distinct $u,u'\in U$, $c(v,v')\in\phi(u,u')$ whenever $\gamma(v)=u$ and $\gamma(v')=u'$ and for $u\in U$, the following occurs: (1) if exactly one of $\{\leftarrow,\rightarrow\}$ is in $\phi(u,u)$, then the oriented edges of $\gamma^{-1}(u)$ are a subdigraph of a transitive tournament (2) if neither $\leftarrow$ nor $\rightarrow$ is in $\phi(u,u)$, then $\gamma^{-1}(u)$ has no oriented edges, (3) if $\nonarrow\not\in\phi(u,u)$, then $\gamma^{-1}(u)$ has no nonedges and (4) if $\unarrow\not\in\phi(u,u)$, then $\gamma^{-1}(u)$ has no undirected edges.
\end{defn}

We define the set of types $\K(\hh)$ that we need to consider for this problem.
\begin{defn}
Let $\hh$ be a hereditary property of $r$-graphs [digraphs].  We use the notation $\F(\hh)$ to be the minimal set of $r$-graphs [digraphs] such that $\hh=\bigcap_{H\in\F(\hh)}\forb(H)$, where $\forb(H)$ denotes the property of having no induced copy of $H$.

We also denote $\K(\hh)$ to be the set of all $r$-types [dir-types], $K$, such that $H\not\mapsto K$ for all $H\in\F(\hh)$.
\end{defn}

The $f_K$ function is what we use to compute the edit distance.
\begin{defn}
For an $r$-type, $K=(U,c)$ on $k$ vertices, and a density vector $\p=(p_1,\ldots,p_r)$, we define the function $f_K(\p)$ as follows: For $\rho=1,\ldots,r$, let the matrix ${\bf A}_{\rho}$ be such that the $(i,j)^{\rm th}$ entry is $1$ if $c(u_i,u_j)\ni\rho$ and zero otherwise.  If ${\bf J}$ denotes the $k\times k$ all-ones matrix, ${\bf 1}$ denotes the $k\times 1$ all-ones vector, then
$$ f_K(\p)=\frac{1}{k^2}{\bf 1}^T\left({\bf J}-\sum_{\rho=1}^rp_{\rho}{\bf A}_{\rho}\right){\bf 1} . $$
\end{defn}

The $f_K$ function is defined in a slightly different way for digraphs.
\begin{defn}
For a dir-type, $K=(U,c)$ on $k$ vertices, and a density vector $\p=(p,q)$, we define the function $f_K(\p)$ as follows: For $\rho=\nonarrow,\unarrow$, let the matrix ${\bf A}_{\rho}$ be such that the $(i,j)^{\rm th}$ entry is $1$ if $c(u_i,u_j)\ni\rho$ and zero otherwise.  The matrix ${\bf A}_{\rightarrow}$ has the property that the $(i,j)^{\rm th}$ entry is
$$ \left\{\begin{array}{rl}
      1, & \mbox{if $c(u_i,u_j)$ contains exactly one member of $\{\leftarrow,\rightarrow\}$;} \\
      2, & \mbox{if $c(u_i,u_j)\supseteq\{\leftarrow,\rightarrow\}$; and} \\
      0, & \mbox{otherwise.}
   \end{array}\right. $$
If ${\bf J}$ denotes the $k\times k$ all-ones matrix, ${\bf 1}$ denotes the $k\times 1$ all-ones vector, then
$$ f_K(\p)=\frac{1}{k^2}{\bf 1}^T\left({\bf J}-(1-p-2q){\bf A}_{\nonarrow}-p{\bf A}_{\unarrow}-q{\bf A}_{\rightarrow}\right){\bf 1} . $$
\end{defn}

The entry of $2$ is necessary in order to account for the fact that fewer editing operations are required if both directions are permitted rather than simply one direction.

Finally, some definitions with respect to edit distance:
\begin{defn}
For $r$-graphs [digraphs] $G=(V,c)$ and $G'=(V,c')$ on the same labeled vertex set, the expression $\dist(G,G')$ counts the number of pairs of vertices $v,v'$ such that $c(v,v')\neq c'(v,v')$.

The distance of $G$ from $\hh$ is $\min\{\dist(G,G') : G'\in\hh\}$.
\end{defn}

We need to express the main application differently in the case of $r$-graphs and digraphs.  However, only the $r$-graph version will be proven.

\begin{thm}\label{thm:app:multicolor}
   Let $G'$ be an $r$-graph in hereditary property $\hh=\bigcap_{H\in\F(\hh)}\forb(H)$ and ${\bf p}=(p_1,\ldots,p_r)$ be a probability vector.  Then, there exists an $r$-type $K\in\K(\hh)$ such that $H\not\mapsto K$ for all $H\in\F(\hh)$ and with probability going to $1$ as $n\rightarrow\infty$, $\dist(G_{n,{\bf p}},\hh)\geq f_K({\bf p})\binom{n}{2}-o(n^2)$.
\end{thm}

\begin{thm}\label{thm:app:digraph}
   Let $G'$ be a digraph in hereditary property $\hh=\bigcap_{H\in\F(\hh)}\forb(H)$ and ${\bf p}=(p,q)$ be a probability vector.  Then, there exists an dir-type $K\in\K(\hh)$ such that $H\not\mapsto K$ for all $H\in\F(\hh)$ and with probability going to $1$ as $n\rightarrow\infty$, $\dist(G_{n,{\bf p}},\hh)\geq f_K({\bf p})\binom{n}{2}-o(n^2)$.
\end{thm}

\begin{proof}
Fix $\eta\gg\delta\gg\epsilon>0$.  Let $G$ be distributed according to $G_{n,{\bf p}}$ and $G'\in\hh$ be a graph of distance $\dist(G,\hh)$ from $G$.  Apply Theorem~\ref{thm:main} with $m=\epsilon^{-1}$ and any decreasing function $\E$ for which $\E(0)=\epsilon$ to $G'$ and consider the partition $\A'=(V_1',\ldots,V_k')$. Construct the $r$-type [dir-type] $K_0=(U,c_0)$ on vertex set $U=\{u_1,\ldots,u_k\}$ as follows.  For distinct $i,j$, $c_0(u_i,u_j)\ni\rho$ if and only if the pair $(V_i',V_j')$ is $\E(k)$-regular such that the color $\rho$ occurs with density at least $\delta$.

Now, we shall define $c_0$ on the vertices; i.e., $c_0(u_i,u_i)$, $u_i\in U$, such that $K_0\in\K(\hh)$. Assume no such assignment to the vertices exists; i.e., for any choice of colors of the vertices, there exists an $H\in\F(\hh)$ for which $H\mapsto K_0$.  Apply the regularity lemma (Theorem~\ref{thm:reglem:multicol} in the $r$-graph case or Theorem~\ref{thm:reglem:digraph} in the digraph case) to each of the clusters $V_i'$ and use Ramsey theory find a clique of miniclusters that are regular with positive density in the same color.  Assign that color to $u_i$ to complete the definition of $K_0$.  Using the relevant slicing and embedding lemmas, we see that if $H\mapsto K_0$, then there is an induced copy of $H$ in $G'$, a contradiction. (See the authors and K\'ezdy~\cite{AKM} for details in the graph case.)

As to counting the number of changes, for all distinct $i<i'$, it is the case that $d_{G',\rho}(V_i,V_{i'})=0$ for all $\rho\not\in c_0(v_i,v_i')$.  By Theorem~\ref{thm:main}, we can look at the equipartition $\A$ and see that for all but $\E(0)k^2$ such pairs, $d_{G',\rho}(V_i,V_{i'})\leq\E(0)$ for all $\rho\not\in c_0(u_i,u_i')$.  Now consider the equipartition $\A$ as applied to $G$.  We see that
\begin{eqnarray*}
   \dist(G,G') & \geq & \sum_{1\leq i<i'\leq k}\sum_{\rho\not\in c_0(u_i,u_{i'})}\left(d_{G,\rho}(V_i,V_{i'})-\E(0)\right)|V_i||V_{i'}| -\E(0)k^2\left\lceil\frac{n}{k}\right\rceil^2 \\
   & \geq & \sum_{1\leq i<i'\leq k}\sum_{\rho\not\in c_0(u_i,u_{i'})}d_{G,\rho}(V_i,V_{i'})\left\lfloor\frac{n}{k}\right\rfloor^2 -\E(0)r\binom{k}{2}\left\lceil\frac{n}{k}\right\rceil^2 -\E(0)k^2\left\lceil\frac{n}{k}\right\rceil^2 .
\end{eqnarray*}

A routine Chernoff bound computation shows that, since $k$ is bounded, if $1\leq i<i'\leq k$, then the probability that $d_{G,\rho}(V_i,V_{i'})<p_{\rho}-\lfloor n/k\rfloor^{-1/3}$ is at most $\exp\left\{-2\lfloor n/k\rfloor^{1/3}\right\}$.  Given an equipartition of $V$ of order $k$, the probability that there exists some pair $(V_i,V_{i'})$ and some $\rho\in\{1,\ldots,r\}$ such that $d_{G,\rho}(V_i,V_{i'})<p_{\rho}-\lfloor n/k\rfloor^{-1/3}$ is at most $r\binom{k}{2}\exp\left\{-2\lfloor n/k\rfloor^{1/3}\right\}$.  The number of equipartitions, disregarding the labeling of the vertices, is bounded by a function of $S=\Smain\left(r,\E(0)^{-1},\E\right)$.  Hence, the probability of having an equipartition with one such pair is $O\left(\exp\left\{-2(n/S)^{1/3}\right\}\right)$.

So, with that probability, and the fact that $\E(0)=\epsilon$,
\begin{eqnarray}
   \dist(G,G') & \geq & \sum_{1\leq i<i'\leq k}\sum_{\rho\not\in c_0(u_i,u_{i'})}p_{\rho}\left\lfloor\frac{n}{k}\right\rfloor^2 -\left\lfloor\frac{n}{k}\right\rfloor^{5/3}r\binom{k}{2} -\epsilon r\binom{k}{2}\left\lceil\frac{n}{k}\right\rceil^2 -\epsilon k^2\left\lceil\frac{n}{k}\right\rceil^2 \nonumber \\
   & = & \frac{1}{2}\sum_{\scriptsize \begin{array}{c} 1\leq i,i'\leq k \\ i\neq i' \end{array}}\sum_{\rho\not\in c_0(u_i,u_{i'})}p_{\rho}\left\lfloor\frac{n}{k}\right\rfloor^2 -\left\lfloor\frac{n}{k}\right\rfloor^{5/3}r\binom{k}{2} -\epsilon r\binom{k}{2}\left\lceil\frac{n}{k}\right\rceil^2 -\epsilon k^2\left\lceil\frac{n}{k}\right\rceil^2 \nonumber \\
   & \geq & \frac{1}{2}\sum_{1\leq i,i'\leq k}\sum_{\rho\not\in c_0(u_i,u_{i'})}p_{\rho}\left\lfloor\frac{n}{k}\right\rfloor^2 -\frac{k}{2}\left\lfloor\frac{n}{k}\right\rfloor^2 -r\binom{k}{2}\left\lfloor\frac{n}{k}\right\rfloor^{5/3} -\epsilon r\binom{k}{2}\left\lceil\frac{n}{k}\right\rceil^2 -\epsilon k^2\left\lceil\frac{n}{k}\right\rceil^2 \nonumber \\
   & = & f_K(\p)\frac{k^2}{2}\left\lfloor\frac{n}{k}\right\rfloor^2 -\frac{k}{2}\left\lfloor\frac{n}{k}\right\rfloor^2 -r\binom{k}{2}\left\lfloor\frac{n}{k}\right\rfloor^{5/3} -\epsilon r\binom{k}{2}\left\lceil\frac{n}{k}\right\rceil^2 -\epsilon k^2\left\lceil\frac{n}{k}\right\rceil^2 \nonumber \\
   & \geq & f_K(\p)\binom{n}{2} \nonumber \\
   & & -\left[\left(\binom{n}{2}-\frac{k^2}{2}\left\lfloor\frac{n}{k}\right\rfloor^2\right) +\frac{k}{2}\left\lfloor\frac{n}{k}\right\rfloor^2 +r\binom{k}{2}\left\lfloor\frac{n}{k}\right\rfloor^{5/3} +\epsilon r\binom{k}{2}\left\lceil\frac{n}{k}\right\rceil^2 +\epsilon k^2\left\lceil\frac{n}{k}\right\rceil^2\right] . \label{eq:error}
\end{eqnarray}

Since $k\geq m=\epsilon^{-1}$, we can see that the error term in (\ref{eq:error}) is $O\left(r\epsilon n^2\right)$.  So, for any $\eta>0$, the probability that $\dist(G_{n,\p},\hh)\geq f_K(\p)\binom{n}{2}-\eta n^2$ goes to 1 as $n\rightarrow\infty$.
\end{proof}

\bibliographystyle{plain}

\begin{thebibliography}{99}
\bibitem{AFKSz} N. Alon, E. Fischer, M. Krivelevich and M. Szegedy, Efficient testing of large graphs, \textit{Combinatorica} \textbf{20}(4) (2000), no. 4, 451--476.
\bibitem{ASh} N. Alon and  A. Shapira, Testing subgraphs in directed graphs, \textit{J. Comput. System Sci.} \textbf{69} (2004), no. 3, 353--382.
\bibitem{AS} N. Alon and  U. Stav, What is the furthest graph from a hereditary property? \textit{Random Structures Algorithms} \textbf{33} (2008), no. 1, pp. 87--104.
\bibitem{AKM} M. Axenovich, A. K\'ezdy and R. Martin, On the editing distance of graphs, \textit{J. Graph Theory} \textbf{58}(2) (2008), pp. 123--138.
\bibitem{KS} J. Koml\'os and M. Simonovits, Szemer\'edi's regularity lemma and its applications in graph theory,  \textit{Combinatorics, Paul Erd\H{o}s is eighty, Vol. 2 (Keszthely, 1993)}, 295--352, Bolyai Soc. Math. Stud., 2, J\'anos Bolyai Math. Soc., Budapest, 1996.
\bibitem{Sz} E. Szemer\'edi, Regular partitions of graphs, \textit{Probl\`emes combinatoires et th\'eorie des graphes (Colloq. Internat. CNRS, Univ. Orsay, Orsay, 1976)}, pp. 399--401, Colloq. Internat. CNRS, 260, \textit{CNRS, Paris}, 1978.
\end{thebibliography}

\end{document}